\documentclass{amsart}

\usepackage{a4,xy,amssymb}
\usepackage{amsthm,amsmath,amscd,amssymb,graphics}
\usepackage[hyperindex,pdfmark]{hyperref}
\xyoption{all}

\makeindex

\def\eps{{\epsilon}}

\def\C{{\mathbb C}}

\def\N{{\mathbb N}}

\def\Q{{\mathbb Q}}

\def\implies{\Rightarrow}

\def\to{\rightarrow}
\def\ot{\leftarrow}
\def\otsym{\overset{\leftarrow}{+}}

\def\<{\langle}
\def\>{\rangle}

\theoremstyle{definition}
\newtheorem{definition}{Definition}[section]
\newtheorem{theorem}{Theorem}[section]
\newtheorem{corollary}{corollary}[section]
\newtheorem*{theorem*}{Theorem}

\newtheorem{lemma}[theorem]{Lemma}
\theoremstyle{remark}
\newtheorem{remark}{Remark}

\newcounter{kd}

\newcommand{\Rep}{\ensuremath{\mathsf{Rep}}}
\newcommand{\stab}{\ensuremath{\mathsf{stab}}}
\newcommand{\DRep}{\ensuremath{\mathsf{DRep}}}

\newcommand{\iss}{\ensuremath{\mathsf{iss}}}
\newcommand{\Diss}{\ensuremath{\mathsf{Diss}}}

\newcommand{\id}{\ensuremath{\mathsf{id}}}
\newcommand{\GL}{\ensuremath{\mathsf{GL}}}
\newcommand{\DGL}{\ensuremath{\mathsf{DGL}}}
\newcommand{\Sp}{\ensuremath{\mathsf{Sp}}}

\newcommand{\Mat}{\ensuremath{\mathsf{Mat}}}

\newcommand{\Stab}{\ensuremath{\mathsf{Stab}}}

\newcommand{\se}[1]{\begin{equation*}\begin{split}#1\end{split}\end{equation*}}

\newcommand {\vtx}[1]{*+[o][F-]{\scriptstyle{#1}}}
\newcommand {\svtx}[1]{*+[o][F.]{\scriptstyle{#1}}}

\newcommand {\Tr}{\mathsf {Tr}}

\author{Raf Bocklandt}
\address{Universiteit Antwerpen (UIA) \\ B-2610 Antwerp (Belgium)}
\email{rafael.bocklandt@ua.ac.be}

\title{A slice theorem for quivers with an involution}
\begin{document}

\maketitle

\begin{abstract}
We study the Luna slice theorem in the case of quivers with an involution
or supermixed quivers as introduced by Zubkov in \cite{Zub}.
We construct an analogue to the notion of a local quiver setting described in \cite{LBP}. 
We use this technique to determine dimension vectors of simple supermixed representations.
\end{abstract}

\section{Introduction}
Given a reductive algebraic group $G$
and a $G$-representation $V$ we can construct the algebraic quotient $V/\!\!/G$, which is
the affine variety corresponding to the ring of invariant polynomial functions $\C[V]^G$.
The embedding $\C[V]^G \subset \C[V]$ gives rise to a quotient map $V \to V/\!\!/G$.

The main problem in invariant theory is to describe the geometry of such a quotients.
There are several questions that one can try to answer: What is the dimension of $V/\!\!/G$?
How do the fibers of $V/\!\!/G$ look like? Is $V/\!\!/G$ a smooth variety?

In complete generality a solution for these problems is unattainable but
given restrictions on the groups or the representations one can expect some interesting partial 
results. In many cases one can find a certain class of couples $(V,G)$ which share the same
geometrical properties for their quotients. More precisely one can try to find classes that are closed under local behavior.
By this we mean that if we have a couple $(V,G)$ and a point $p \in V/\!\!/G$ we can find another couple $(V_p,G_p)$
of the same class such that there is an \'etale neighborhood of $p$ that is locally isomorphic to an \'etale neighborhood 
of the zero point in $V_p/\!\!/G_p$. Such a result simplifies the questions a lot because we can use this local result to 
reduce the questions about complicated representations to more simple representations.

For $(G,V)$ a representation space of a quiver this was done by Procesi and Le Bruyn in \cite{LBP}.
Similar results have been obtained for representation spaces of preprojective algebras by Crawley-Boevey \cite{Boevey}.

In this paper we will study the case of supermixed quivers. These were introduced and studied by Zubkov and Lopatim in \cite{Zub},\cite{Zub2},\cite{Zublopat} and are closely related to generalized quivers which were studied by Derksen and Weyman in \cite{DW}. First we will give a
coordinate free description of a representation space of a supermixed quiver by means of involutions on semisimple algebras.
Then we will extend the results by Derksen and Weyman to obtain a representation  theoretic interpretation of the points in the representation spaces and in the quotient. This will enable us to formulate an extension of the result on local quivers
by Procesi and Le Bruyn  to supermixed quivers. To make full use of this result we will also determine which supermixed settings
have simple representations.

In the rest of this paper all our varieties and algebra will be considered over $\C$ but the results apply to every algebraicly closed
field of characteristic zero.

\section{A quick review of quiver settings}

We briefly recall that 
a quiver $Q$ consists of a set of vertices $Q_0$, a set of arrows $Q_1$ and two maps $h,t:Q_1 \to Q_0$ 
which assign to each arrows its source and its tail.
It is easy to see that that a quiver is uniquely defined by its Euler form.

A dimension vector of a quiver is a map $\alpha: Q_0\to \N$. We will call a couple $(Q,\alpha)$ a quiver setting.

To every quiver setting we can associate a semisimple algebra, its standard left module
\[
 S_\alpha := \bigoplus_{v \in Q_0}\Mat_{\alpha_v\times \alpha_v}(\C),~\C^\alpha:= \bigoplus_{v \in Q_0}\C^{\alpha_v}
\]
and an $S_\alpha$-bimodule
\[
 \Rep(Q,\alpha) =  \bigoplus_{a \in Q_1}\Mat_{\alpha_{h(a)}\times \alpha_{t(a)}}(\C).
\]
Vice versa if $S$ is a semisimple algebra and $M$ an $S$-bimodule we can find a quiver setting $(Q,\alpha)$ 
unique up to isomorphism such that $(S,M) \cong (S_\alpha,\Rep(Q,\alpha))$.
The vertices in $Q$ correspond to the maximal set of orthogonal idempotents $\{v_1,\dots v_k\}$ in the center of $S$. The dimension vector is $\alpha_{v_i}= \sqrt{\dim v_iS}$ while the number of arrows from $v_j$ to $v_i$ is $\frac{\dim v_iMv_j}{\alpha_{v_i}\alpha_{v_j}}$.

We define the group $\GL_{\alpha}$ as the group of invertible elements in $S_\alpha$. This group
has an action on $\Rep(Q,\alpha)$ by conjugation: $m \mapsto gmg^{-1}$. This action has a categorical 
quotient:
\[
 \iss(Q,\alpha) = \Rep(Q,\alpha)/\!\!/\GL_{\alpha}.
\]
In algebraic terms we construct the quotient as follows: let $\C[\Rep(Q,\alpha)]$ be the ring of 
polynomial functions over $\Rep(Q,\alpha)$. On this ring we have an action of $\GL_{\alpha}$ coming 
from conjugation action on $\Rep(Q,\alpha)$. The subring of functions that are invariant under this action
is the ring of polynomial function over the categorical quotient:
\[
 \C[\iss(Q,\alpha)] = \C[\Rep(Q,\alpha)]^{\GL_{\alpha}}.
\]
A path of length $k$ in a quiver is a sequence of arrows $p=a_1\dots a_k$ with $t(a_i)=h(a_{i+1})$. We denote
its head and tail as $h(p)=h(a_1)$ and $t(p)=t(a_k)$. A path the tail of which equals its head is called
a cycle. A vertex is also called a path of length zero.
A cycle is called primitive if the heads of all its arrows are different i.e. it runs through every vertex at most once.
The path algebra $\C Q$ is the vector space with as basis all paths and its multiplication is the concatenation of paths if possible and zero if not. The path algebra is graded by the length of the paths and it is  Morita equivalent to the tensor algebra $T_S M=S \oplus M  \oplus M \otimes_S M \oplus\cdots$. 

The degree zero part of $\C Q$ can be embedded into $S$ as the center: $\C Q_0=Z(S)$. As such we can see $\C^\alpha$ as a left $\C Q_0$-module.
Any point $W \in M=\Rep(Q,\alpha)$ can be identified with a left module of $\C Q$ such that the restriction to $\C Q_0$ is $\C^\alpha$. Every arrow $a$ will act as $W_a$ between $\C^{\alpha_{t(a)}}$ and $\C^{\alpha_{h(a)}}$ and as zero between the rest. Therefore the points of $\Rep(Q,\alpha)$ will be called representations of $Q$ with dimension vector $\alpha$. 

On the other hand if $W$ is a left $\C Q$-module we can find bases for every space $vW, v \in Q_0$. If we express the action of arrows $a:t(a)W \to h(a)W$ as matrices according to these bases we get a representation of $Q$ with dimension vector $\alpha_W: v \mapsto \dim vW$. 

In this way we can speak of simple, semisimple and
indecomposable representations of $Q$. Two points $W$ and $W'$ will give isomorphic representations if and only if they are in the same $\GL_{\alpha}$-orbit.

If $W$ is a representation we can evaluate the path $p$: $W_p=W_{a_1}\dots W_{a_k}$.
To every cycle $c$ we can associate the map $f_c :\Rep(Q,\alpha) \to \C: W \mapsto \Tr  W_c$.
This map is invariant under the $\GL_{\alpha}$-action and in general every invariant map can be 
written in terms of these:
\begin{theorem}[Le Bruyn-Procesi]
 The ring $\C[\iss(Q,\alpha)]$ is generated by functions of the form $f_c$ where $c$ is a cycle in $Q$.
\end{theorem}

The representation-theoretical interpretation of the quotient can be summarized as follows.
\begin{theorem}
 The points in $\iss(Q,\alpha)$ are in one to one correspondence to isomorphism classes
 of the semisimple representations of $Q$ in $\Rep(Q,\alpha)$. Two points in $\Rep(Q,\alpha)$
are mapped to the same point in $\iss(Q,\alpha)$ if and only if they have the same semisimplification.
\end{theorem}

\section{Dualizing structures and supermixed settings}

Let $S$ be a finite dimensional semisimple algebra  and let $M$ be an $S$-bimodule.
\begin{definition}
A \emph{dualizing structure} on $(S,M)$ consists of two linear involutions $*: S\to S$ and
$*:M \to M$ and a bilinear form $\<,\>:\C^\alpha \times \C^\alpha\to \C$, satisfying the following compatibility relations:
\begin{itemize}
 \item $\forall a,b \in S_{\alpha}: (ab)^*=b^*a^*$.
 \item $\forall a,b \in S_{\alpha}:\forall m\in M:(amb)^*=b^*m^*a^*$.
\end{itemize}
The \emph{dualizing group} of $S$ is the group of elements for which the inverse and the involution
coincide:
\[
D(S) = \{ g \in S : g g^*=1\}
\]
while the \emph{dualizing subspace} of $M$ is the subspace
\[
D(M) = \{v\in M:v^*=v\} 
\]
The group $D(S)$ has an action on $D(M)$ by conjugation: $v \mapsto gvg^*$ because
$(gvg^*)^* =gvg^*$. 
\end{definition}

We can turn this data into the language of quivers. 
First we are going to impose the dualizing structure on a quiver setting.
An involution $*$ on $S_{\alpha}$ restricts to an involution of the center
which is a ring-automorphism. This maps idempotents to idempotents so we get an involution $\phi$ on the set of vertices $\{v_i\}$.

Given $\phi$ we can construct a standard involution on $S_\alpha$:
\[
 s \mapsto s^\dagger\text{ with }(s^\dagger)_v = s_\phi(v)^\top 
\]
where the transpose is taken according to the identification $S_\alpha = \oplus_{v \in Q_0}\Mat_{\alpha_v\times \alpha_v}(\C)$.
The composition of $*$ and $\dagger$ gives us an automorphism of $S_\alpha$ which is internal
because all idempotents are fixed. Therefore we can say that there exists a $g \in S_\alpha$ such that
\[
(s^*)_v = g_{v} s_{\phi(v)}^\top  g_{v}^{-1}
\]
The fact that $*^2 = \id$ implies that
\[
 ((s^*)^*)_v = g_{v} g_{\phi(v)}^{-\top}s_{v} g_{\phi(v)}^{\top} g_{\phi(v)}^{-1}=s_v
\]
so $g_v=\eps_v g_{\phi(v)}^\top $ for some scalars $\eps_v$ and $\eps_v\eps_{\phi(v)}=1$.
Because we only use the $g_v$ for conjugation they are determined up to a scalar themselves.
This allows us to chose $\eps_v=\pm 1$ and hence we also have that $\eps_v=\eps_{\phi(v)}$.

Now we will tackle the dualizing structures on $M=\Rep(Q,\alpha)$.
We suppose that the involution on $S_\alpha$ is in the form above and that $\phi$ is the corresponding involution
on the vertices $Q_0$.
Choose a second involution $\phi: Q_1\to Q_1$ such that $h(\phi(a))=\phi(t(a))$ and if $h(a)=\phi(t(a))$ we demand
that $\phi(a)=a$.
Given this we can construct a 'standard' dualizing structure on $\Rep(Q,\alpha)$.
\[
 m \mapsto m^\dagger\text{ with }(m^\dagger)_a = \eps_{t(a)} g_{h(a)} m_{\phi(a)}^\top g_{t(a)}^{-1}.
\]

As in the case of $S_\alpha$, we can compose the involution $*:\Rep(Q,\alpha)\to \Rep(Q,\alpha)$ we want to study with
this standard involution to obtain an automorphism $*\dagger$ of $\Rep(Q,\alpha)$ as an
$S_\alpha$-module. By Schur's lemma there are coefficients $\sigma_{ab},~a,b \in Q_1$ such that
$(m^{*\dagger})_a= \sum_{a \in Q_1} \sigma_{ab}m_b$. Also note that $\sigma_{ab}$ is only nonzero if
$h(a)=h(b)$ and $t(a)=t(b)$.

By taking linear combinations of the arrows, we can diagonalize $\sigma$ and this gives the following formula for $*$:
\[
(m^*)_a = \sigma_a \eps_{t(a)} g_{h(a)} m_{\phi(a)}^\top g_{t(a)}^{-1}
\]
The fact that $*$ is an involution implies that $\sigma_a\sigma_{\phi(a)}=1$ and rescaling ensures us that we can suppose $\sigma_a = \pm 1$.

All this rewriting can be summarized in the terminology of supermixed quivers.

\begin{definition}
 A supermixed quiver $\Q$ consists of a quiver $Q$, two involutions $\phi$ (one on the vertices and one on the arrows) and two sign maps $\eps:Q_0 \to \{\pm 1\}$ and $\sigma: Q_1 \to \{\pm 1\}$.
which satisfy:
\begin{itemize}
 \item[M1] $h(\phi(a))=\phi(t(a))$,
 \item[M2] $h(a)=h(\phi(a)) \implies a=\phi(a)$,
 \item[M3] $\eps_v\eps_{\phi(v)}=\sigma_a\sigma_{\phi(a)}=1$.
\end{itemize}
We say that $(\Q,\gamma)$ is a supermixed setting if $\gamma: Q_0 \to \sqcup_{n \in \N} \GL_n(\C)$ is a map that assigns
to each vertex a matrix such that $g_v=\eps_vg_{\phi(v)}^\top$. To $\gamma$ we can naturally associate a dimension vector which assigns
to every vertex $v$ the dimension of $\gamma_v$. We will denote this dimension vector by $\bar\gamma$, but we will leave the bar in cases
where the dimension vector is just used as an index f.i. we will write $S_{\gamma}$ instead of $S_{\bar \gamma}$.

Every supermixed quiver setting defines involutions $*$ on the algebra $S_\gamma$ and on the space $\Rep(\Q,\gamma)$:
\se{
(s^*)_v &= \gamma_{\phi(v)}s_{\phi(v)}^\top \gamma_{\phi(v)}\\
(m^*)_a &= \sigma_a \eps_{h(a)} \gamma_{t(a)}^\top m_{\phi(a)}^\top \gamma_{h(a)}
}
The couple $(S_{\gamma},\Rep(\Q,\gamma))$ with the involutions above is called the dualizing structure coming from the supermixed quiver setting.

A vertex with $v=\phi(v)$ is called \emph{orthogonal} if $\eps_v=1$ and \emph{symplectic} if $\eps_v=-1$,
a vertex for which $v\ne \phi(v)$ is called \emph{general}.

An arrow with $a=\phi(a)$ is called \emph{symmetric} if $\sigma_a=1$ and \emph{antisymmetric} if $\sigma_a=-1$, an arrow
for which $a\ne \phi(a)$ is called \emph{general}.
\end{definition}
The terminology introduced above is an adaptation from the one introduced by Zubkov in \cite{Zub}. Instead of partitioning the vertices
into sets of size one and two and looking at the double quiver we use involutions. Also we use the $\gamma$ for the identification between
the vector space and their duals instead of the more explicit approach in \cite{Zub}. The translation between the two notations is
straightforward, but our adaptation makes it easier to describe the results we obtained concerning local quiver settings.

\begin{theorem}
Every couple $(S,M)$ with a dualizing structure is isomorphic to a dualizing structure coming from
a normalized supermixed quiver setting. 
\end{theorem}
\begin{proof}
This follows from the discussion above.
\end{proof}

There is a redundancy in the definition of the supermixed settings. 
\begin{theorem}
Two supermixed settings $(\Q_1,\gamma_1)$ and $(\Q_2,\gamma_2)$ have isomorphic dualizing structures if
the underlying dimension vectors are the same and the sign maps have identical values on the fixed vertices and arrows.
\se{
\bar \gamma_1 &=\bar \gamma_2 \\
v =\phi(v) &\implies \eps_1(v)=\eps_2(v)\\
a =\phi(a) &\implies \sigma_1(a)=\sigma_2(a)
}
\end{theorem}
\begin{proof}
If we apply a base change $h \in S_{\alpha}$ then $\gamma_v$ will transform as
\[
  (\gamma')_v= h_v \gamma_v h_{\phi(v)}^\top.
\]
This means that if $v=\phi(v)$ we can transform $\gamma_v$ to the identity matrix if $\eps_v=1$ and to
the standard symplectic matrix $\Lambda_{2n}= \left(\begin{smallmatrix}0&-1\\1&0\end{smallmatrix}\right)^{\oplus n}$ if $\eps_v=-1$.
If $v\ne \phi(v)$ we can change $\gamma_v$ to $\id$ (and hence $\gamma_{\phi(v)}$ to $\eps_v \id$).
So we see that the isomorphism class of the involution does not depend on the exact nature of $\gamma$ only on $\bar\gamma$ and $\eps$.

We can also multiply the $\gamma_v$ with scalars  
\[
  (\gamma')_v= \lambda_v \gamma_v 
\]
without changing the involution. This changes  $\eps_v$ into $\eps'_v = \eps_v \lambda_v/\lambda_{\phi(v)}$,
so only the values of $\eps$ on the fixed vertices are important. Keep in mind that $\sigma_a$ also changes to
$\sigma_a \lambda_{(h(\phi a))}/\lambda_{h(a)}$.

Finally we can alter the sign of $\sigma_a$ if $a\ne \phi(a)$ by a base change on the vertices. 
\end{proof}

We can change $\eps$,$\sigma$ and $\gamma$ without changing
the isomorphism class of $(S_{\gamma},\Rep(\Q,\gamma))$ such that
\begin{itemize}
 \item $v \ne \phi(v) \implies \eps_v=1$, 
 \item $a \ne \phi(a) \implies \sigma_a=1$, 
 \item $\gamma_v= \id_{\bar \gamma_v}$ if $\eps_v=1$ and $\gamma_v= \Lambda_{\bar \gamma_v}$ if $\eps_v=-1$.
\end{itemize}
If these conditions are met we speak of a \emph{strict} supermixed setting, one can prove
that there is a one to one correspondence between strict supermixed settings and isomorphism classes of dualizing structures.
In the next sections, however we will need a more flexible description of dualizing structures so we will
also allow non-strict settings.

The dualizing group of a supermixed quiver setting $(\Q,\gamma)$ is by definition the group
\[
 \{ h \in S_{\gamma}| \forall v \in Q_0: h_v \gamma_vh_{\phi(v)}^\top \gamma_{v}^{-1}=1 \}
\]
because up to isomorphism we can bring the supermixed setting in its strict form and we can conclude that
\[
 \DGL_{\gamma} := D(S_\gamma) \cong \prod_{\{v,\phi(v)\} \in Q_0} \begin{cases}
O_{\gamma_v} &v=\phi(v) \text{ and } \eps_v=1 \\
Sp_{\gamma_v} &v=\phi(v) \text{ and } \eps_v=-1\\
GL_{\gamma_v}&v\ne \phi(v).
\end{cases}
\]
This explains the terminology for the vertices.

We will depict supermixed quiver settings in the following way. We will consider the $*$ operation as a reflection around
the vertical axis. In this way it is obvious to identify the duals of the vertices and arrows. The arrows with $\sigma_a=-1$ and
the vertices with $\eps_v=-1$ will be drawn with dashed lines. If needed we will put the expression of $\gamma_v$ inside the vertex
but if we don't do this we suppose that $\gamma_v=\id$ if $v=\phi_v$ and $\eps_v=1$ or $v\ne \phi$ and $v$ is the leftmost vertex of the pair $(v,\phi_v)$. If $v=\phi(v)$ and $\eps_v=-1$ we suppose that $\Lambda_{\bar \gamma_v}$.

The two supermixed quiver settings below describe isomorphic dualizing structures, but only the left setting is strict.

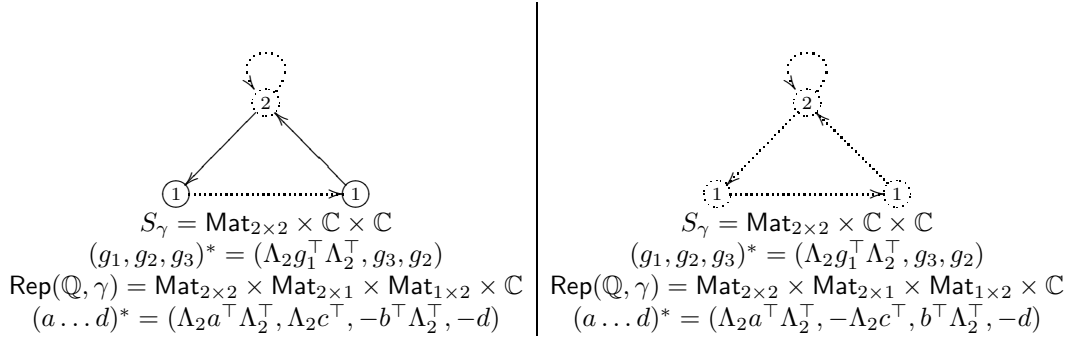
\begin{tabular}{c|c}
$\xymatrix{
&&\\
&\svtx{2}\ar@{.>}@(ru,lu)\ar[dl]&\\
\vtx{1}\ar@{.>}[rr]&&\vtx{1}\ar[ul]
}$
&
$\xymatrix{
&&\\
&\svtx{2}\ar@{.>}@(ru,lu)\ar@{.>}[dl]&\\
\svtx{1}\ar@{.>}[rr]&&\svtx{1}\ar@{.>}[ul]
}$
\\
$S_{\gamma}=\Mat_{2\times 2} \times \C\times \C$&
$S_{\gamma}=\Mat_{2\times 2} \times \C\times \C$\\
$(g_1,g_2,g_3)^*=(\Lambda_2 g_1^\top \Lambda_2^\top, g_3,g_2)$&
$(g_1,g_2,g_3)^*=(\Lambda_2 g_1^\top \Lambda_2^\top, g_3,g_2)$\\
$\Rep(\Q,\gamma)= \Mat_{2\times 2}\times \Mat_{2\times 1} \times \Mat_{1\times 2}\times \C$&
$\Rep(\Q,\gamma)= \Mat_{2\times 2}\times \Mat_{2\times 1} \times \Mat_{1\times 2}\times \C$\\
$(a\dots d)^* = (\Lambda_2 a^\top \Lambda_2^\top, \Lambda_2 c^\top, -b^\top\Lambda_2^\top, -d)$&
$(a\dots d)^* = (\Lambda_2 a^\top \Lambda_2^\top, -\Lambda_2 c^\top, b^\top\Lambda_2^\top, -d)$
\end{tabular}

\section{$\eps$-mixed modules}

Given a quiver setting $(Q,\alpha)$ one can interpret the spaces $\Rep(Q,\alpha)$ and $\iss(Q,\alpha)$ as classifying the 
the $\C Q$-module structures that one can put on $\C^\alpha$. We will now try to do a similar thing for supermixed 
quiver settings. The discussion below matches the discussion 
of symmetric quivers in \cite{DW}.

Given a a supermixed quiver setting $(\Q,\gamma)$ we define a bilinear form on the space $\C^{\bar\gamma}:= \bigoplus_v \C^{\bar\gamma_v}$:
\[
\<x,y\> := \sum_v x_v \gamma_v y_{\phi(v)}^\top.
\]
If we identify $\eps$ with the element $\sum_v\eps_v v \in \C Q$ we see that this form has a special property: it is $\eps$-commuting
\[
\forall x,y \in \C^{\gamma}: \<x,y\> =\<y,\eps x\>.
\]
This bilinear form can be used to define an adjoint: if $\varphi: \C^\gamma \to \C^\gamma$ we set
\[
\forall x,y \in \C^\alpha: \<\varphi x,y\> = \<x,\varphi^\natural y\>
\]
Every $W \in \Rep(\Q,\gamma)$ gives us a left $\C Q$-module structure $\rho_W:\C Q \to \mathsf{End}(\C^{\bar \gamma})$ on $\C^{\bar \gamma}$.
The adjoint allows us to express the action on the dual module: $\rho_W^\natural$. This is a right module so we cannot
see this dual module as a point in $\Rep(\Q,\gamma)$.

To do this we put an anti-automorphism $*$ on the path algebra 
\[
*:\C Q \to \C Q: \begin{cases}
                  v \to \phi(v)\\
		  a \to \sigma_a \eps_{h(a)}\phi(a)
                 \end{cases}
\text{ and }\forall x,y \in \C Q:(xy)^*=y^*x^*.
\]
This anti-automorphism is not an involution but it satisfies
\[
 \forall x \in \C Q: x^{**} =\eps x \eps
\]

This anti-automorphism allows us to turn left $\C Q$-modules into right $\C Q$-module: $\rho_*(u) := \rho(u^*)$.
Therefore, given the left module $W \in \Rep(\Q,\gamma)$, we can consider the left module $W^* \in \Rep(\Q,\gamma)$ corresponding to $(\rho_W^\natural)_*$. This gives us an involution on $\Rep(\Q,\gamma)$ and it is easy to check that this is in fact 
the same as the involution coming from the original dualizing structure. 

All this warrants the following definition:
an \emph{$\eps$-mixed module} is a left $\C Q$-module
$V$ together with a nondegenerate bilinear map $\<,\>: V \times V \to \C$ satisfying
\begin{itemize}
 \item $\forall v,w \in V:\<v,w\>=\<w,\eps v\>$ ($\eps$-commuting),
 \item $\forall v,w \in V:\forall x \in \C Q: \<xv,w\>=\<v,x^*w\>$ ($*$-compatible).
\end{itemize}
If the $*$-compatibility only holds for $\C Q_0$ instead of the whole $\C Q$ we will speak of \emph{almost $\eps$-mixed modules}.

For every (almost) $\eps$-mixed $\C Q$-module we can choose bases $(b_i^v)$ in every $vW$
and use these to construct a supermixed quiver setting $(\Q,\gamma)$ with
\[
(\gamma_v)_{ij}=\<b_i^v,b_j^v\>
\]
The $\eps$-commutativity of $\<,\>$ ensures that whatever bases we take we will end up with a
supermixed quiver setting that gives the same dualizing structure.

According to these bases we can express every $a$ as a matrix $W_a$. 
If $W$ is $\eps$-mixed then $(W_a)$ is in fact an element of $\DRep(\Q,\gamma)$. 
Vice versa if $W \in \DRep(\Q,\gamma)$ we can build an $\eps$-mixed module out of it in the usual way.

Furthermore two (almost) $\eps$-mixed modules are
called $\eps$-isomorphic if there is a module morphism between them that preserves the bilinear form.
The group of these isomorphisms is $\DGL_\gamma$ and
it is easy to check that $V,W \in \DRep(\Q,\gamma)$ correspond to $\eps$-isomorphic modules if and only
if they are in the same orbit under $\DGL_\gamma$. We can conclude:
\begin{theorem}
The $\DGL_{\gamma}$-orbits in $\DRep(\Q,\gamma)$ classify the $\eps$-mixed modules structures on $\C^\gamma,~\<,\>$ up to $\eps$-isomorphism.
\end{theorem}

A second theorem is less trivial, but its proof can easily be adapted from \cite{DW}[Thm 2.6].
\begin{theorem}
Two $\eps$-mixed modules are isomorphic as $\eps$-mixed modules if and only if they are isomorphic as $\C Q$-modules.
\[
 \forall m \in \DRep(\Q,\gamma): \DGL_{\gamma} m = \GL_{\alpha} m \cap \DRep(\Q,\gamma)
\]
\end{theorem}
The $\eps$-mixed structure is not only compatible with isomorphisms but it is also compatible with degenerations:
\begin{theorem}\label{semisimp}
If $V$ is an $\eps$-mixed module 
then its semisimplification as a $\C Q$-module is also an $\eps$-mixed module.
\end{theorem}
\begin{proof}
We prove this by induction on the length of the composition series of $V$. If $V$ is simple then the statement is trivially true. Now suppose $S\subset V$ is a simple submodule then we also have a projection between the dual modules $V=V^*\to S^*$. 

The kernel of this projection is $S^\perp=\{v \in V| \forall w \in S:\<v,w\>=0\}$. 
There are two possibilities: $S\cap S^\perp=0$ or $S\subset S^\perp$. In the first case $\<,\>_S$ and $\<,\>_{S^\perp}$ are nondegenerate and
$V=S\oplus S^\perp$. By the induction hypothesis $S^{\perp ss}$ admits an $\eps$-mixed structure and hence $V^{ss}=S\oplus S^{\perp ss}$ as well.

In the second case $\<,\>_{S^\perp}$ is degenerate but it becomes nondegenerate if we quotient out $(S^\perp)^\perp=S$. 
By the induction hypothesis $(S^{\perp}/S)^{ss}$ admits an $\eps$-mixed structure we can use this structure to put an $\eps$-mixed structure
on $V^{ss}= S\oplus (S^\perp/S)^{ss} \oplus S^*$:
\[
\<(s_1,t_1,u_1),(s_2,t_2,u_2)\>=u_2(s_1)+u_1(s_2)+\<t_1,t_2\>_{(S^\perp/S)^{ss}}
\]
It is easy to see that the action of $\C Q$ is compatible with this form.
\end{proof}

Apart from the orbits themselves we are also interested in the quotient space of the orbits.
\[
\Diss(\Q,\gamma) = \DRep(\Q,\gamma)/\!\!/\DGL_\gamma
\]
The main result we can use for this was proven by Zubkov in \cite{Zub} and is stated in terms of invariants
\begin{theorem}
The $\DGL_{\gamma}$-invariant functions on $\DRep(\Q,\gamma)$
come from $\GL_{\gamma}$-invariant function on $\Rep(Q,\gamma)$: the diagram below is commutative
\[
\xymatrix{
\C[\Rep(Q,\gamma)]\ar[r] &\C[\DRep(\Q,\gamma)]\\
\C[\Rep(Q,\gamma)]^{\GL_{\gamma}}\ar[r]\ar@{^(->}[u] &\C[\DRep(\Q,\gamma)]^{\DGL_{\gamma}}\ar@{^(->}[u]}
\]
\end{theorem}

The relation between the quotient space and the representation theory is stated in this corollary
\begin{corollary}
The points in $\Diss(\Q,\gamma)$ parameterize the semisimple $\eps$-mixed module up to isomorphism.
\end{corollary}
\begin{proof}
Every fiber of the map $\DRep(\Q,\gamma) \to \Diss(\Q,\gamma)$ contains a semisimple representation: suppose $W$ is in the fiber of $p$ and
$W^{ss}$ is the semisimplification of $W$. 
Because of theorem \ref{semisimp} $W^{ss}$ can be considered as sitting inside $\DRep(\Q,\gamma)$.
As it is the semisimplification of $W$ it has the same values as $W$ on the $\GL_{\gamma}$-invariant functions and
hence also on the $\DGL_{\gamma}$-invariant functions. As a consequence $W$ and $W^{ss}$ are in the same fiber.
Also two non-isomorphic semisimples cannot be in the same fiber because there orbits are closed (they are intersections of closed $\GL_{\gamma}$-orbits with $\DRep(\Q,\gamma)$.
\end{proof}

An $\eps$-submodule is a submodule $W \subset V$ such that $\<,\>|_{W}$ is non degenerate,
an $\eps$-mixed module is \emph{$\eps$-irreducible} if it has no nontrivial supermixed submodules. It is clear that it is possible to define
the direct sum of two $\eps$-mixed modules.
If $W\subset V$ is an $\eps$-submodule, we can define the submodule $W^\perp:=\{v \in V| \forall w \in W:\<v,w\>=0\}$.
The nondegeneracy of $W$ makes that $W \cap W^\perp=0$ and the compatibility of $\<,\>$ with the involution makes that $W^\perp$
is also an $\eps$-mixed representation, so $V=W\oplus W^\perp$.

It is not true that an $\eps$-irreducible module is always semisimple as a $\C Q$-module, so its orbit
might not be closed in $\DRep(\Q,\gamma)$. An example of this phenomenon is the supermixed quiver setting:
\[
\xymatrix{
\vtx{1}\ar[r]&\vtx{1}
}
\]
All $\eps$-mixed representations of this setting are $\eps$-irreducible but only the one that assigns zero to the arrow is semisimple.
Hence it is more interesting to restrict our attention to
$\eps$-irreducible that are also semisimple. Such modules will be called \emph{$\eps$-simple} and every
semisimple $\eps$-mixed module is a direct sum of $\eps$-simple modules.

\begin{theorem}
Suppose that $V$ is $\eps$-simple. There are $3$ possibilities:
\begin{enumerate}
 \item $V$ is a simple $\C Q$-module, 
 \item $V\cong W\oplus W^*$ where $W$ is a simple $\C Q$-module and $W$ and $W^*$ are isomorphic modules. 
 \item $V\cong W\oplus W^*$ where $W$ is a simple $\C Q$-module and $W$ and $W^*$ are non-isomorphic modules.
\end{enumerate}
\end{theorem}
\begin{proof}
If we are not in the first case, there exists a proper simple submodule $W \subset V$.
As $W^\perp\cap W$ is a submodule of $W$, $W^\perp\cap W=W$
and $W$ is perpendicular to itself
(it cannot be that $W^\perp\cap W=0$ because then $W$ would be an $\eps$-simple submodule).

Because $V$ is semisimple we know that $V\cong W\oplus W^\perp/W \oplus W^*$.
But the bilinear form on $W^\perp/W$ is nondegenerate so $W^\perp/W$ is an $\eps$-submodule of $V$ and hence $0$.
This proves that there are $3$ possible cases. 
\end{proof}
We will call these types of $\eps$-simple modules (1) orthogonal, (2) symplectic and (3) general,
because the $\eps$-automorphism groups of such modules are (1) $O_1$, (2) $\Sp_2$, (3) $\GL_1$.

The theorem implies that we can decompose a semisimple $\eps$-mixed representation as follows
\[
\bigoplus_{1 \le \ell \le k_1} S_\ell^{e_\ell} \bigoplus_{k_1 < \ell \le k_2} (S_\ell \oplus S_\ell^*)^{e_\ell} \bigoplus_{k_2 < \ell \le k_3}  (S_\ell \oplus S_\ell^*)^{e_\ell}
\]
Where the $(S_\ell)_{1 \le \ell \le k_1}$ are orthogonal, the $(S_\ell\oplus S_\ell^*)_{k_1 < \ell \le k_2}$ symplectic and the $(S_\ell \oplus S_\ell^*)_{k_2 < \ell \le k_3}$ are general.

To end this section we will show that there is a close connection between orthogonal, symplectic and general representations
\begin{theorem}\label{symplectic}
\begin{enumerate}
\item[]
\item 
If $V \cong W \oplus W^*$ is a symplectic $\eps$-simple representation of $\Q$ then
we can give $W$ the structure of an $\eps^-$-mixed representation of $\Q^-$ that is {\bf also orthogonal $\eps^-$-simple} 
Here $\Q^-$ is the supermixed quiver with the same underlying quiver and involutions as $\Q$ but
with $\eps^-=-\eps$ and $\sigma^- =-\sigma$.
\item
If $V \cong W \oplus W^*$ is a general $\eps$-simple representation of $\Q$ with $\alpha_W=\alpha_{W^*}$ 
then we can give $W$ the structure of an {\bf almost $\eps^+$-mixed} simple representation of $\Q^+$ that is {\bf not isomorphic to an $\eps^+$-mixed}.
Here $\Q^+$ is the supermixed quiver with the same underlying quiver and involutions as $\Q$ but
with $\eps^+: v \mapsto 1$ and $\sigma^+: a \mapsto \sigma_a \eps_{t(a)}$.
\end{enumerate}
\end{theorem}
\begin{proof}
\begin{enumerate}
 \item 
Because $W\cong W^*$, every element in $V$ can then be written as a couple $(v,v')$ with $v,v' \in W$.
We know that $W^\perp=W$ and $W^{*\perp}=W^*$ therefore there are two bilinear forms on $W$ such that
\[
\<(v,v'),(w,w')\>_V = \<v,w'\>_1 + \<v',w\>_2
\]
The $\eps$-commutativity of $\<,\>_V$ implies that
\[
 \<v',w\>_2 = \< w , \eps v'\>_1
\]
We denote the adjoints according to $\<,\>_1$ by $\natural$ and $\flat$:
\[
 \forall v,w \in W: \<v,\varphi w\>_1 = \<\varphi^\flat v, w\>_1 \text{ and } \<\varphi v,w\>_1 = \< v,\varphi^\natural w\>_1.
\]
Note that the $\eps$-commutativity of $\<,\>$ also implies that $\eps^\flat=\eps^\natural=\eps^{-1}$.

If $a \in \C Q$ then the right adjoint of 
$\rho_V(a)=\rho_W(a)\oplus \rho_W(a)$ according to $\<,\>_V$ is
\[
(\eps \rho_W(a)\eps)^\flat \oplus \rho_W(a)^\natural= \rho_V(a^*) = \rho_W(a^*) \oplus \rho_W(a^*) 
\]
so we know that $\eps \rho_W(a)^\flat\eps =\rho_W(a)^\natural=\rho_W(a^*)$.

Taking twice the $V$-adjoint shows us that $\rho_W(a)^{\flat\flat}=\eps\rho_W(a)\eps$.

It is not necessarily so that $\<,\>_1$ is $\eps$-commuting, so define $\phi$ by the equation $\<v,w\>=\<\phi w,v\>$. This map has
the property that $\phi^\flat\phi=1$. Now taking twice the adjoint here tell us that $\phi\rho_W(a)^{\flat\flat}\phi^{-1}=\rho_W(a)=\eps \rho_W(a)^{\flat\flat} \eps^{-1}$. 
 This holds for all $a \in \C Q$ and
as $\rho_W$ is simple we can conclude that $\phi$ must be $\pm \eps$ so $\<,\>_1$ is either $\eps$-commuting or $-\eps$-commuting.

The former is impossible because $V$ is $\eps$-simple and the diagonal submodule $\nabla = \{(v,v)|v \in W\}$ is nondegenerate:
\[
 \<(v,v),(w,w)\>_V = \<v,w\>_1 + \< w,\eps v\>_1 = \<v,w\> +  \<\eps w, \eps v\>_1 = 2\<v,w\>.
\]
We can conclude that $W$ admits a $-\eps$-structure. But if we want to consider $W$ as an $-\eps$-mixed representation we also have to change 
$\sigma$ to $-\sigma$ because otherwise the $*$-operation on $\C Q$ is altered. 
We can conclude that $W$ is a $-\eps$ representation of the supermixed quiver $(Q,\phi,-\eps,-\sigma)$.

\item
Chose a bases $(b_{vi})$ for every $vW$ and construct the dual bases $(b^*_{vi})$ in $vW^*$ such that $\<b_{vi},b^*_{wj}\>_V = \delta_{v\phi(w)}\delta_{ij}$. 
We also define a bilinear form on $W$ by $\<b_{vi},b_{wj}\>_W = \delta_{v\phi(w)}\delta_{ij}$. This form is only nondegenerate if
if $\alpha_W=\alpha_{W^*}$. It is also $\eps^+$-commuting and it trivially true that $W$ is an almost $\eps^+$-mixed module.
If $\natural$ is the adjoint according to $\<,\>_W$ we can write the adjoint of $V$ as
\[
 \rho_V(a^*) = (\eps \rho_{W^*}(a) \eps^{-1})^\natural  \oplus \rho_{W}(a)^\natural.
\]
So we can see that $\rho_{W}(a^*)=(\eps \rho_{W^*}(a) \eps^{-1})^\natural$. If $W$ were isomorphic to an $\eps^+$-mixed module
then $\rho_{W}(a^*)\cong \rho_{W}(a^*)^\natural$ but this is impossible because 
$\rho_{W}(a^*)\cong \rho_{W^*}(a^*)^\natural$ and $\rho_W$ is not isomorphic to $\rho_{W^*}$ as $V$ is general.
\end{enumerate}
\end{proof}

\section{Local mixed quivers}

Another tool we want to adapt to the $\eps$-mixed quiver case is the Luna slice theorem
and the construction of local quivers.

First of all let us recall the Luna slice theorem \cite{Luna}. We will restrict to its use for group actions on a vector space. 
Let $V$ be a vector space with a linear action from an algebraic group $G$. If $v \in V$ has an orbit
$Gv$ which is closed in $V$, we can approximate the quotient $V/\!\!/G$ in an \'etale neighborhood of $v$ as follows. 
Construct the normal space which is the quotient of $V$ by the tangent space to the orbit 
\[
 N_v = V/T_vGv.
\]
On this space there is an action of $G_v$ (the stabilizer of $v$ in $G$) because it acts on both
$V$ and $T_vGv$.

\begin{theorem}[Luna Slice] 
There exists an \'etale neighborhoods $U_v$ of $v \in V$ and $U_0$ of $0 \in N_v$ such that we have the following commutative diagram
\[
 \xymatrix{
U_v\ar[d]^\cong\ar@{->>}[r]&U_v/\!\!/G \ar[d]^\cong\ar[r]^{et}&V /\!\!/G\\
U_0\times_{G_v}G\ar@{->>}[r]&U_0/\!\!/G_v \ar[r]^{et}&N_v /\!\!/G_v.
}
\]
Which means that the quotient of $V$ by $G$ around $v$ is locally isomorphic to the quotient
of $N_v$ by $G_v$ around the zero.
\end{theorem}

Now suppose $W$ is a semisimple representation of $Q$ in $\Rep(Q,\alpha)$. We can write down the decomposition
of $W$ as a direct sum of simples:
\[
W \cong S_1^{\oplus e_1}\oplus \cdots \oplus S_k^{\oplus e_k}
\]
The stabilizer of $W$ in $\GL_{\alpha}$ is isomorphic to $\GL_{e_1}\times \dots \times \GL_{e_k}$.
Putting these things together we get the local quiver theorem:

\begin{theorem}[Le Bruyn-Procesi]\label{local}
For a point $p \in \iss(Q,\alpha)$ corresponding to a semisimple representation $W=S_{1}^{\oplus e_1}\oplus \dots \oplus S_{k}^{\oplus e_k}$,
there is a quiver setting $(Q_p,\alpha_p)$ called the \emph{local quiver setting} such that
we have an isomorphism between an \'etale open
neighborhood of the zero representation in $\iss_{\alpha_p} Q_p$ and an \'etale open neighborhood of $p$.

$Q_p$ has $k$ vertices corresponding to the set $\{S_i\}$ of simple factors of $W$ and the number of arrows from $S_j$ to $S_i$ equals
\[
\sum_{a \in Q_1}\alpha^i_{h(a)}\alpha^j_{t(a)} - \sum_{v \in Q_0}\alpha^i_v\alpha^j_{v}+\delta_{ij}
\]
where $\alpha^i$ is the dimension vector of the simple component $S_i$.

The dimension vector $\alpha_p$ maps every simple $S_i$ to its multiplicity $e_i$.
\end{theorem}

We can extend this theorem to supermixed settings. First we note that if $W=W^*$ then the stabilizer of $W$ 
is closed under the involution.
\[
 g \in \Stab_W \implies gWg^{-1}=W \implies g^{-1*}Wg^*=W \implies g^{-1*} \in \Stab_W \implies g^* \in \Stab_W.
\]
The same holds for the tangent space to the orbit $T_W \GL_{\gamma}W$ 
\se{
(T_W \GL_{\gamma}W)^* &= \{ sW-Ws|s \in S_\alpha\}^*\\
&= \{ -s^*W^*+W^*s^*|s \in S_\gamma\}\\
&= \{ -s^*W+Ws^*|-s^* \in S_\gamma\}=T_W \GL_{\gamma}W.
}
And hence we can transport the involution on $\Rep(Q,\gamma)$ to the normal space
$\Rep(Q,\gamma)/T_W \GL_{\gamma}W$.

\begin{theorem}\label{epslocal}
For a point $p \in \Diss(\Q,\gamma)$ corresponding to a semisimple representation 
\[
W = \bigoplus_{1 \le \ell \le k_1} S_\ell^{e_\ell} \bigoplus_{k_1 < \ell \le k_2} (S_\ell \oplus S_\ell^*)^{e_\ell} \bigoplus_{k_2 < \ell \le k_3}  (S_\ell \oplus S_\ell^*)^{e_\ell}
\]
there is a supermixed quiver setting $(\Q_p,\gamma_p)$ called the \emph{local mixed quiver setting} such that
we have an isomorphism between an \'etale open
neighborhood of the zero representation in $\Diss(\Q_p,\gamma_p)$ and an \'etale open neighborhood of $p$.

$Q_p$ has $2k_3-k_2$ vertices corresponding to the set $Q_{p0}=\{S_i,S_i^*\}/{\cong}$ of isomorphism classes of simple factors of $W$.
The new mixing factor gives the vertices in $Q_{p0}$ the same type as the corresponding $\eps$-mixed representations:
\se{
\eps_p(S_i)=\eps(S_i^*)&=\begin{cases}
                         -1& i\in [k_1+1,k_2] \\
                         +1& i\not \in [1,k_1]\cup[k_2+1,k_3]
                        \end{cases}\\
\gamma_p(S_i)=\gamma_p(S_i^*)&=\begin{cases}
                         \Lambda_{2e_\ell}& i\in [k_1+1,k_2] \\
                         \id_{e_\ell}& i\not \in [1,k_1]\cup[k_2+1,k_3]
                        \end{cases}
}
The number of arrows from $Y$ to $X\in Q_{p0}$ equals
\[
A_{XY} := \sum_{a \in Q_1}\alpha^X_{h(a)}\alpha^Y_{t(a)}-\sum_{v \in Q_0}\alpha^X_{v}\alpha^Y_{v}+\delta_{XY}
\]
If $X=Y^*$ the number of symmetric (antisymmetric arrows) equals
\[
\frac{A_{XY} \pm\left(\sum_{a=\phi(a)}\sigma_a\alpha^X_{h(a)})+\sum_{v=\phi(v)}\eps_v \alpha^X_v - \eps_X\delta_{XY}\right)}2
\]
The $+$-sign is used for symmetric arrows and the $-$-sign for antisymmetric arrows. 
In these expressions $\alpha^X$ is the dimension vector associated to the representation $X$.
\end{theorem}
\begin{proof}
Suppose $W=W^*$ is a semisimple representation of $(\Q,\gamma)$ with decomposition
\se{
W &=  \bigoplus_{1 \le \ell \le k_1} S_\ell^{e_\ell} \bigoplus_{k_1 < \ell \le k_2} (S_\ell \oplus S_\ell^*)^{e_\ell} \bigoplus_{k_2 < \ell \le k_3}  (S_\ell \oplus S_\ell^*)^{e_\ell}\\
&= \bigoplus_{1 \le \ell \le k_1} \C^{e_\ell}\otimes S_\ell \bigoplus_{k_1 < \ell \le k_2} \C^{e_\ell}\otimes(S_\ell \oplus S_\ell^*) \bigoplus_{k_2 < \ell \le k_3}  \C^{e_\ell}\otimes (S_\ell \oplus S_\ell^*).
}
Every $S_\ell$ has a dimension vector which we will denote by $\alpha^\ell$. From the previous section we know that $S_\ell$ itself is $\eps$-commuting
for $\ell \le k_1$ so we can also find a $\gamma^\ell$ with $\bar \gamma^\ell=\alpha^\ell$. The same can be done for $S_\ell$ if $k_1<\ell\le k_2$, only
we have to keep in mind that these $\gamma^\ell$ will be $\eps^-$-commuting. Finally we set $\gamma^\ell:v \to \id_{\alpha^\ell_v}$ and
$\gamma^{\ell*}:v \to \eps_v\id_{\alpha^\ell_v}$ if $\ell > k_2$.
These will not give us a supermixed setting because it might be that $\alpha_v \ne \alpha_{\phi(v)}$.

These $\gamma^\ell$ can be used to construct the $\gamma$ for $W$:
\[
\gamma = \bigoplus_{1\le \ell \le k_1} \id_{e_{\ell}}\otimes \gamma^{\ell} \bigoplus_{k_1 < \ell \le k_2} \begin{pmatrix}0&-\id_{e_{\ell}}\\\id_{e_{\ell}}&0 \end{pmatrix}\otimes \gamma^\ell
\bigoplus_{k_2 < \ell \le k_3}\id_{\ell}\otimes \begin{pmatrix}0&\gamma^{\ell}\\\gamma^{\ell^*}&0 \end{pmatrix}
\]
If we introduce $\eps_p$ and $\gamma_p$ as they are defined in the theorem,
we can rewrite
\[
\gamma = \bigoplus_{1\le \ell \le k_2} \gamma_p(S_\ell)\otimes \gamma^\ell 
\bigoplus_{k_2 < \ell \le k_3}
\begin{pmatrix}0&\gamma_p(S_\ell)\otimes\gamma^\ell\\
\gamma_p(S_\ell^*)\otimes\gamma^{\ell^*}&0 \end{pmatrix}.
\]
The basis for which this expression holds has the form $b_{\ell\mu}\otimes c_{\ell v \nu}$ where $b_{\ell\mu}$ is the basis for $\C^{e_\ell}$
while $c_{\ell v \nu}$ forms the basis for $S_\ell (l\le k_1)$ or $(S_{\ell}\oplus S_{\ell}^*)$. We indexed the latter by two
indices: the vertices in $\Q$ and for every $v$ $\nu \in [1,\dots,\bar \gamma^\ell_v (+\gamma^{\ell*}_v)]$.
To every basis element corresponds a projection operator which we will denote by $\beta_{\ell\mu}\otimes \varsigma_{\ell v \nu} \in S_\gamma$.

Let $C(W)=\{a \in S_{\gamma}| aW=Wa\}$ be the centralizer of $W$.
By Schur's lemma and the fact that $\C Q$ only acts on the right hand part of the basis $b_{\ell\mu}\otimes c_{\ell\nu}$, we can conclude
that $C(W)$ looks like
\[
 \prod_{1\le \ell \le k_1}\Mat_{e_\ell} \otimes \id_{S_\ell}\prod_{k_1 < \ell \le k_2}\Mat_{2e_\ell} \otimes \id_{S_\ell}\prod_{k_2 < \ell \le k_3}\Mat_{e_\ell} \otimes \id_{S_\ell} \times \Mat_{e_\ell} \otimes \id_{S_\ell^*} 
\]
which is isomorphic with $S_{\gamma_p}$. The involution is also isomorphic to the one on $S_{\gamma_p}$:
\se{
s^*
&= \left(\bigoplus A_\ell \otimes \id_{S_{\ell}} \bigoplus_{\ell} B_\ell  \otimes \id_{S_{\ell}}   \bigoplus C_\ell  \otimes \id_{S^*_{\ell}} \right)^*\\
&= \bigoplus \gamma_{p\ell} A_\ell^\top \gamma_{p\ell}^{-1} \otimes \id_{e_{\ell}} \bigoplus_{\ell} C_\ell^\top\otimes\id_{S^*_\ell} \bigoplus B_\ell^\top  \otimes \id_{S_\ell}.
}.

To calculate the arrows in the local quiver we first need a lemma that deals with restriction of dualizing structures
\begin{lemma}
Let $(S,M)$ be any dualizing structure such that $S$ is isomorphic to the $S_{\gamma}$ from above and denote
the mixed quiver to which $M$ corresponds $Q_M$.
The structure of the restricted dualizing structure $(S_{\gamma_p},M)$ corresponds to a new quiver $Q_M'$ with
$2k_3-k_2$ vertices and
every arrow $a$ in $Q_M$ transforms for every pair of vertices $X,Y \in Q_{p0}$  into
$\alpha^X_h(a)  \alpha^Y_t(a)$ arrows from $Y$ and $X$. 
If $a$ is an (anti)-symmetric arrow and $X=Y^*$ then 
\[\frac{{\alpha^X_{h(a)}}^2 + \alpha^X_{h(a)}}2\] 
are of the same type as $a$ 
while the rest of the arrows ($\frac{{\alpha^X_{h(a)}}^2 - \alpha^X_{h(a)}}2$) are of the opposite type
of $a$.
\end{lemma}
\begin{proof}
We are interested in
the structure of the restricted dualizing structure $(S_{\alpha_p},M)$. Let $a$ be an arrow in $M$ and denote
its corresponding simple sub-$S_{\alpha}$-bimodule by $M_a$. As an $S_{\gamma_p}$ bimodule $M_a$ decompose as a direct sum of 
$S_{\gamma_p}$-bimodules
\se{
M_a &= 1 M_a 1^* \text{ ($1=\bigoplus \beta_{\ell\mu}\otimes\varsigma_{\ell v \nu}=\bigoplus \id_{e_\ell}\otimes\varsigma_{\ell v \nu}$)}\\
&=
\bigoplus_{\ell_1,\ell_2}\bigoplus_{\mu_1,\mu_2}  ( \id_{e_{ \ell_1}}\otimes\varsigma_{\ell_1 h(a)\mu_1}) M_a ( \id_{e_{\ell_2}}\otimes\varsigma_{\ell_2 t(a)\mu_2})^* 
}
Note that we only need the $\varsigma$'s for which the vertex $v=h(a),t(a)$ because the others act as zero on $M_a$.
All these components are simple $S_{\gamma_p}$-bimodules and hence
represent arrows from the vertex $Y=S_{\ell_2}$ to $X=S_{\ell_1}$.
So in total there are 
\[
\sum_a \alpha^{X}_{h(a)}  \alpha^{Y}_{t(a)}
\]
between $Y$ and $X$ in $Q'$.

Under the involution the only components which are mapped onto themselves
are the ones for which $a=\pm a^*, \ell_1=\ell_2(=\ell)$ and $\mu_1=\mu_2(=\mu)$.

The action under the involution on $x \in (\id_{e_\ell}\otimes \varsigma_{\ell h(a)\mu}) M_a (\id_{e_\ell}\otimes \varsigma_{\ell h(a)\mu})^*$ is given by 
\se{
x  \mapsto & \sigma_a \eps_{t(a)} \gamma  \left((\id_{e_\ell}\otimes\varsigma_{\ell h(a)\mu}) x (\id_{e_\ell}\otimes\varsigma_{\ell h(a)\mu})^*\right)^{-1}  g\\
&= \sigma_a \eps_{t(a)} \underbrace{(\id_{e_\ell}\otimes\varsigma_{\ell h(a)\mu}) \gamma  (\id_{e_\ell}\otimes\varsigma_{\ell h(a)\mu})^{*T}}_{\eps_{h(a)}\eps_{S_{\ell}}(\gamma_{\ell}\otimes \varsigma_{\ell h(a)\mu})} x^\top  
\underbrace{(\id_{e_\ell}\otimes\varsigma_{\ell h(a)\mu})^\top  \gamma^{-1} (\id_{e_\ell}\otimes\varsigma_{\ell h(a)\mu})^{*}}_{(\gamma_{\ell}^{-1}\otimes \varsigma_{\ell h(a)\mu}^*)}\\
&=  \sigma_{a} \underbrace{\eps_{h(a)}\eps_{t(a)}}_{=1}\eps_{S_{\ell}}  (\gamma_{\ell}\otimes \id) x (\gamma_{\ell}^{-1}\otimes \id) 
}
We can conclude that $(\alpha^{X}_{h(a)})^2-\alpha^{X}_{h(a)}$ of the arrows in which $a=\phi(a)$ decomposes 
are not mapped to themselves under the involution $M_b\ne M_b^*$. Taking linear combinations $b-b^*$ and $b+b^*$ we can make this  
equivalent to $\frac{(\alpha^{X}_{h(a)})^2-\alpha^{X}_{h(a)}}2$ symmetric and an equal number antisymmetric arrows.

The calculation above adds to these equal amounts $\alpha^{X}_{h(a)}$ more symmetric or antisymmetric depending on the type of $a$.
\end{proof}

As $S_{\gamma_p}$-bimodules with dualizing structure we have that $N_p = \Rep(\Q,\gamma)/T_W \GL_{\gamma} W$. Also we can identify
$T_W \GL_{\gamma} W$ with $S_{\gamma}/S_{\gamma_p}$ but this identification only works as bimodules without the dualizing structure. In order
to make it work with the dualizing structure we must put a new dualizing structure on $S_{\gamma}$ as an $S_{\gamma}$-bimodule. This is done by putting
$\star : S_\gamma \to S_{\gamma}: x \to -x^*$. This new structure turns the morphism
\[
\pi : S_{\gamma} \to T_W \GL_{\gamma} W: x \mapsto xW - Wx
\]
into a $*$-morphism: $\pi(x^\star) = x^\star W - Wx^\star = -x^*W+Wx^* =-x^*W^*+W^*x^*= \pi(x)^*$.
In quiver terminology $Q_{S_{\gamma}}$ is a quiver with the same number of vertices as $Q$ but with 
a unique loop in every vertex. The loops in orthogonal vertices are antisymmetric, while those
in symplectic vertices are symmetric.

So to determine the total number of arrows in $N_p$ from $Y$ to $X$ we can use the following formula
\se{
A_{XY}&:= \#\{X \ot  Y \text{ in $Q'_{N_p}$} \}\\
&= \#\{X \ot Y \text{ in $Q'_{\Rep(\Q,\gamma)}$} \}
- \#\{X \ot Y \text{ in $Q'_{S_{\gamma}}$} \}
+ \#\{X \ot Y \text{ in $Q'_{S_{\gamma_p}}$} \}\\
&= \sum_{a \in Q_1} \alpha^{X}_{h(a)}\alpha^{Y}_{t(a)} - \sum_{v \in Q_0} \alpha^{X}_{v}\alpha^{Y}_{v} + \delta_{XY}
}
To determine the number of symmetric arrows from $Y$ to $X$ we can do the same thing:
\se{
\#\{X \otsym Y \text{ in $Q'_{N_p}$} \}
&= \#\{X \otsym Y \text{ in $Q'_{\Rep(Q,\gamma)}$} \}
- \#\{X \otsym Y \text{ in $Q'_{S_{\gamma}}$} \}
+ \#\{X \otsym Y \text{ in $Q'_{S_{\gamma_p}}$} \}\\
&=\frac{A_{XY}+\sum_{a=\phi(a)}\sigma_a \alpha^{X}_{h(a)}- \sum_{v=\phi(v)}(-\eps_{v})\alpha^{X}_{v} + (-\eps_X)\delta_{XY}}2
}
for the antisymmetric arrows we obtain
\[
\frac{A_{XY}-\sum_{a=\phi(a)}\sigma_a \alpha^{X}_{h(a)}+ \sum_{v=\phi(v)}(-\eps_{v})\alpha^{X}_{v} - (-\eps_X)\delta_{XY}}2
 \]
\end{proof}

We will illustrate the theorem with a few examples.
\begin{itemize}
 \item Consider the setting 
\[
\xymatrix{
\vtx{2}\ar@{:>}@/^/[r]&\vtx{2}\ar@{:>}@/^/[l]
}
\]
where the two vertices are each others dual and all arrows are antisymmetric.
If we consider a representation for which every arrow is represented by $\Lambda_2$. This representation is
the direct sum of two isomorphic representation with dimension vector $1$. These two are each others dual so the local quiver setting is
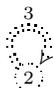
\[
\xymatrix{
\svtx{2}\ar@{:>}@(lu,ru)^3
}
\]
with $3$ loops of which there are $\frac{3+(-4+1)}2=0$ symmetric arrows and $3$
\item
The setting
\[
\xymatrix{
\vtx{2}\ar@2@/^/[d]\\
\svtx{2}\ar@2@/^/[u]
}
\]
where the upper vertex is orthogonal and the lower vertex is symplectic. Now we look at the local quiver setting
of the representation that assigns to the two left arrows the identity matrices (and to the right arrows $-\Lambda_2$).

The stabilizer of this setting is $GL_1 \times GL_1$ and the local quiver setting looks like
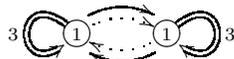
\[
\xymatrix{
\vtx{1}\ar@{=>}@(lu,ld)_3\ar@{.>}@/^/[r] \ar@/^/@<1ex>[r]&\vtx{1}\ar@{.>}@/^/[l] \ar@/^/@<1ex>[l]\ar@{=>}@(ru,rd)^3
}
\]
~\\
As one can calculate the number of symmetric arrows from the left vertex to the right is $\frac{2 +0-0+0}2=1$
\end{itemize}

\section{Simples}

In this section we are going to 
describe a method to determine whether a given supermixed quiver setting $(\Q,\gamma)$ has a dualizing space,
$\DRep(\Q,\gamma)$, that contains simples and if it does
what kind of simples.

To answer this question we first need to recall when an ordinary quiver setting whose representation space simple representations.
This question was solved by Le Bruyn and Procesi in \cite{LBP}. To state this result we need some terminology.

A quiver setting $(Q,\alpha)$ is \emph{sincere} if all vertices have a nonzero dimension (i.e. $\forall v \in Q_0:\alpha_v>0$).
A quiver $Q$ is \emph{strongly connected} if for every pair of vertices $v,w \in Q_0$ there are paths $v\ot w$ and $w \ot v$.
The \emph{support} of a setting is the subquiver that contains only the vertices with $\alpha_v>0$. The support of a representation
is the same as the support of its setting.

\begin{theorem}[Le Bruyn, Procesi]\label{simple}
Let $(Q, \alpha)$ be a sincere quiver setting.
There exist simple representations of dimension vector $\alpha$ if and only if
\begin{itemize}
\item
If $Q$ is of the form
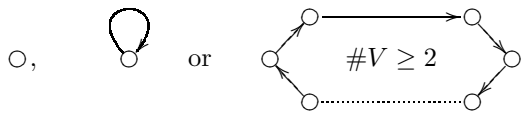
\[
\vcenter{\xymatrix@=.5cm{\vtx{~}}},\hspace{1cm}
\vcenter{\xymatrix@=.5cm{\vtx{~}\ar@(lu,ru)}}\hspace{.5cm}\text{ or }\hspace{.5cm}
\vcenter{\xymatrix@=.2cm{
&\vtx{~}\ar@{->}[rr] &&\vtx{~}\ar@{->}[rd] &\\
\vtx{~}\ar@{->}[ru] &&\#V\ge 2&&\vtx{~}\ar@{->}[ld] \\
&\vtx{~}\ar@{->}[lu] &&\vtx{~}\ar@{..}[ll]&\\
}}
\]
and $\alpha = 1$ (this is the constant map from the vertices to $1$).
\item
$Q$ is not of the form above, but strongly connected and
\[
\forall v \in V: \alpha_v \le \sum_{h(a)=v} \alpha_{t(a)} \text{ and }\alpha_v \le \sum_{t(a)=v} \alpha_{h(a)}
\]
\end{itemize}
If $(Q,\alpha)$ is not sincere, the simple representations classes
are in bijective correspondence to the simple representations classes
of its support.
\end{theorem}

In all cases except for the one vertex without loops there are an infinite number of isomorphism classes of simples with that dimension vector.
In the case of the one vertex $v$ without loops, there is one unique simple representation
$T_v$.

As stated above we have a nice numerical criterion to check whether a given setting has simples or not. The question
for supermixed quiver settings is a bit more complicated because we have $3$ types of simples. 
However it suffices to check for every $(\Q,\gamma)$ there exist orthogonal simples and if so
whether all simples in $\Rep(\Q,\gamma)$ are isomorphic to an orthogonal $\eps$-simple.
\begin{theorem}
The following pairs of statements are equivalent:
\begin{itemize}
 \item $\DRep(\Q,\Lambda_2\otimes \gamma)$ contains symplectic $\epsilon$-simples.
 \item $\DRep(\Q^-,\gamma)$ contains orthogonal $\epsilon$-simples.
\end{itemize}
 and
\begin{itemize}
 \item $\DRep(\Q,\gamma)$ contains general $\epsilon$-simples.
 \item One of the two possibilities below holds:
\begin{enumerate}
 \item $\gamma \cong \id_2\otimes \gamma'$ and $\Rep(\Q^+,\gamma')$ contains simples that are not isomorphic to an orthogonal $\eps+$-simple.
 \item There exists a dimension vector $\alpha$ with $\alpha+\alpha^*=\bar \gamma$ such that $\Rep(Q,\alpha)$ contains simples 
and $\alpha \ne \alpha^*$.
\end{enumerate}
\end{itemize}
\end{theorem}
\begin{proof}
These statements follow from theorem \ref{symplectic}. 
\end{proof}

Given the solution to the first question, (are there orthogonal simples), the answer to the second questions is quite straitforward:
\begin{theorem}
If $\DRep(\Q,\gamma)$ contains orthogonal simples then every simple in $\Rep(\Q,\gamma)$ is isomorphic to an orthogonal simple
if and only if 
\[
\dim \iss(\Q,\gamma) =\dim \Diss(\Q,\gamma).
\]
This condition can be expressed numerically as
\[
\sum_{a \in Q_1} \alpha_{h(a)} \alpha_{t(a)} - \sum_{v \in Q_0} \alpha_v^2 +1 = \frac 12 \left( \sum_{a \in Q_1} \alpha_{h(a)} \alpha_{t(a)} - \sum_{v \in Q_0} \alpha_v^2 + \sum_{\phi(a)=a}\sigma_a \alpha_{h(a)} - \sum_{\phi(v)=v}\eps_{v} \alpha_{v}\right).
\]
with $\alpha=\bar \gamma$.
\end{theorem}
\begin{proof}
Because $\Diss(\Q,\gamma)$ is a closed subset of $\iss(\Q,\gamma)$, the dimension condition implies that both spaces are identical. This means that
every representation in $\Rep(\Q,\gamma)$ is isomorphic to a representation in $\DRep(\Q,\gamma)$. A simple representation that is identical to
its dual is an orthogonal simple. 

Now vice versa suppose every simple in $\iss(\Q,\gamma)$ is isomorphic to an orthogonal then we know that as the simples form a open subset of
$\iss(\Q,\gamma)$, the dimension of $\Diss(\Q,\gamma)$ must be at least the dimension of $\iss(\Q,\gamma)$.

The numerical condition comes from the calculations
\se{
\dim \iss(Q,\alpha)&= \dim \Rep(Q,\alpha) - \dim GL_{\alpha} + \dim \stab(\text{simple})\\
&=\sum_{a \in Q_1} \alpha_{h(a)} \alpha_{t(a)} - \sum_{v \in Q_0} \alpha_v^2 +1\\
\dim \Diss(\Q,\gamma)&=\dim \DRep(\Q,\gamma) - \dim \DGL_{\gamma} + \dim \stab(\text{orth. simple})\\
&=\frac 12 \left( \sum_{a \in Q_1} \alpha_{h(a)} \alpha_{t(a)} + \sum_{\phi(a)=a}\sigma_a \alpha_{h(a)} - \sum_{v \in Q_0} \alpha_v^2 - \sum_{\phi(v)=v}\eps_{v} \alpha_{v}\right).
}
\end{proof}

The ideal answer to the existence question for simples would be a numerical criterion like the one from \cite{LBP}. This becomes 
a combinatorial nightmare because the three different vertex types give different conditions. Therefore we will give an algorithmic approach.
To clarify what is meant by that we will first transform the ordinary theorem into an algorithmic result. 

Before we proceed to prove this algorithm works we need a lemma
\begin{lemma}\label{localsimple}
If $W$ is a semisimple representation in $\Rep(Q,\alpha)$ and $(Q_L,\alpha_L)$ is its local quiver setting then
$\Rep(Q,\alpha)$ contains simple representations if and only if $\Rep(Q_L,\alpha_L)$ contains simples
\end{lemma}
\begin{proof}
The simples in $\Rep(Q,\alpha)$ or in
 $\Rep(Q_L,\alpha_L)$ form a Zariski open part and the \'etale isomorphism will map simples to simples because it preserves
the closedness of orbits and the size of the stabilizers.
\end{proof}
If $W$ is not a representation with dimension vector $\alpha$ but $\alpha^W\le \alpha$ then we can still consider a local quiver setting of
\emph{the completed representation} $\hat W$ which is obtained by adding the standard simples
\[
\hat W := W \oplus \bigoplus_{v \in Q_0}T_v^{\oplus \alpha_v-\alpha^W_v}.
\]

\begin{theorem}[reformulation of theorem \ref{simple}]
To check whether $\Rep(Q,\alpha)$ contains simples follow the following algorithm. 
\begin{itemize}
 \item[S0]
If $(Q,\alpha)$ is not sincere then $\Rep(Q,\alpha)$ contains simples if and only if $\Rep(Q',\alpha')$ does. Here $(Q',\alpha')$ is support
of $(Q,\alpha)$.
 \item[S1]
If $\forall v \in Q_0: \alpha_v\ne 0$ and $Q$ is not strongly connected then $\Rep(Q,\alpha)$ does not contain simples.
If $Q$ has one vertex and no arrows then $\Rep(Q,\alpha)$ contains no simples if $\alpha>1$.
 \item[S2]
If $\alpha=1$ then $\Rep(Q,\alpha)$ contains simples if and only if $Q$ is strongly connected.
\item[S3]
If $\alpha\ne 1$ and $Q$ is strongly connected, then by S2, there exists a simple $W$ with dimension vector $1$. 
$\Rep Q,\alpha$ contains simples if and only if the local quiver setting $(Q_L,\alpha_L)$ of the completed representation $\hat W$
contains simples.
\end{itemize}
\end{theorem}
\begin{proof}
The step S0 is trivial. Step S1 follows from the fact that every strongly connected component gives rise to a subrepresentation. 
If $Q$ is strongly connected every $\alpha=\id$-dimensional representation that assigns to every arrow a nonzero scalar is simple (S2).
Step S3 is a consequence of lemma \ref{localsimple}

Finally we have to show that the recursion will end in a finite number of steps. This is because every $S3$-step reduces the dimension of all vertices with one and adds one vertex with dimension $1$.
\end{proof}

Now we want to generalize the theorem above to the case of orthogonal simples.
First we can generalize the lemma about the local quivers:
\begin{lemma}\label{localmixedsimple}
If $W$ is a semisimple $\eps$-mixed representation in $\DRep(\Q,\gamma)$ and $(Q_L,\gamma_L)$ is its local quiver setting then
$\DRep(Q,\gamma)$ contains simple representations if and only if $\DRep(\Q_L,\gamma_L)$ contains simples
\end{lemma}
\begin{proof}
Analogous to the ordinary case.
\end{proof}

We also need some lemmas about supermixed settings with special dimension vectors.

\begin{lemma}\label{cykel}
Suppose $\Q$ is a supermixed quiver that can be written as $c \cup c^*$ where $c$ is a cycle.
Then we can find an $\eps$-mixed  $\eps$-simple representation $W$ such that its dimension vector
is $1$ or $2$.
\end{lemma}
\begin{proof}
First assume that $c \cap c^*$ is empty then all arrows and vertices are of general type.
Consider the representation $W$ that assigns to every arrow the scalar $1 \in \Mat_{1\times 1}(\C)$.
This representation is the direct sum of two simples corresponding to the two cycles. The involution
swaps these two simples so $W$ is a general $\eps$-simple.

Now assume that $c \cap c^*$ is non-empty but that it does not contain 
symplectic vertices or antisymmetric arrows. In that case we can again assign the scalar $1 \in \Mat_{1\times 1}(\C)$.
to every arrow. This representation will be simple because it is strongly connected and hence $W$ is an orthogonal $\eps$-simple.

Finally if $\Q$ contains a symplectic vertex or an antisymmetric arrow,
We chose a minimal set of arrows $C$ such that $C$ meets every $\{a,\phi(a)\}$.
Let $W$ be the representation that assigns the $2\times 2$-identity matrix to every arrow of $C$ that is not 
selfdual. To the selfdual arrows we assign the identity matrix if $a^*=a$ and the standard symplectic if $a^*=-a$.
$W$ is semisimple because all arrows are invertible and $C \cup \phi(C)$ is strongly connected.
The stabilizer $\DGL_W$ of $W$ inside $\DGL_{\gamma}$ contains
\[
\{ g \in \DGL_{\gamma}| \forall v,w \in \Q'_0:g_v=g_w \in O_2 \cap \Sp_2\} \cong \GL_1
\]
but because all arrows are invertible the value of $g$ on a vertex $v$ fixes all other values:
\[
 g_w = W_{p}g_vW_{p}^{-1} \text{ if $p: v \ot w$},
\]
so $\DGL_W \subset \DGL_{\bar \gamma_v}=\Sp_2$ if $v$ is the symplectic vertex or 
$\DGL_W \subset \{ h \in \DGL_{\bar \gamma_{h(a)}}| h\Lambda_2 h^{T} \}=\Sp_2$ if $a^*=-a$. 

This means that there are two possibilities for $\DGL_W$: either it is $\GL_1$ in which case $W$ is a general $\eps$-simple
or it is $\Sp_2$ in which case $W$ is a symplectic $\eps$-simple.
\end{proof}
\begin{remark}\label{note}
Note that in all cases except for the quiver with just one (symplectic) vertex and only antisymmetric loops, the corresponding local quiver setting has
less vertices of dimension $2$ than the original.
\end{remark}

\begin{definition}
Suppose $(\Q,\gamma)$ is a strict supermixed quiver setting such that 
\begin{enumerate}
 \item 
$\forall v \in Q_0: 1\le \bar\gamma_v\le 2$. 
\item
$\bar \gamma_v=2$ if $v$ is a symplectic vertex or $v=h(a),t(a)$ for an antisymmetric arrow $a$.
\item
No cycle has dimension bigger than $1$.
\item $\Rep(\Q,\gamma)$ contains simples.
\end{enumerate}
Such a supermixed quiver settings are called \emph{basic} and they will replace the quiver settings with dimension vector $1$ in the ordinary case.
\end{definition}

\begin{lemma}\label{21}
If $(Q,\gamma)$ is basic then $\DRep(\Q,\gamma)$ contains orthogonal $\eps$-simples if and only if for every pair of vertices $(v,w)$ of
dimension one there is a path $p: v \to w$ that is not antisymmetric $p^*\ne -p$.
\end{lemma}
\begin{proof}
The condition is necessary because if there is no such path from $v$ to $w$ then the $\C Q v W$ is a subrepresentation but $wW \cap \C Q v W=0$
because every path from $v$ to $w$ is antisymmetric and hence evaluates to zero.

Now we prove that the condition is also sufficient.
First let us define some open subset of $\DRep(\Q,\gamma)$.
\begin{itemize}
 \item 
For every primitive path $p$ ($\ne -p^*$) between two vertices of dimension one we can consider the open subset $U_p \subset \DRep(\Q,\gamma)$ of all
representations $W$ such that $W_p$ is invertible. These subsets are nonempty: there is a surjective $S_{\gamma}$-bimodule morphism 
\[
\begin{cases}
\Rep(\Q,\gamma) \to L_1=\C^2: W \mapsto (W_p,W_{p^*})&p\ne p^*\\
\Rep(\Q,\gamma) \to L_2=\C: W \mapsto (W_p)&p=p^* 
\end{cases}
\]
In both cases the dualizing space is nonzero ($D(L_1)=\{(\lambda,\lambda)|\lambda \in \C\}$, $D(L_2)=\C$), so $U_p$ must also be nonzero.
\item 
For every vertex $v$ of dimension $2$ we can consider the open subset $U_v \subset \DRep(\Q,\beta)$ of all
representations $W$ such that both $\oplus_{h(a)=v}W_a$ and $\oplus_{t(a)=v}W_a$ have rank $2$. It is easy to see that $U_v$ is also non-empty.
\end{itemize}

Suppose that $W \in \cap_p U_p \cap_v U_v$ then $W$ is a simple representation.
If $W'\subset W$ is a simple subrepresentation then it cannot be concentrated in one vertex $v$ with $\bar \gamma_v=2$ because $W$ is inside the $U_v$.
It also cannot be concentrated in only vertices with dimension $\alpha_v=2$ because there is no cycle of this kind of vertices.
So its support contains a vertex of dimension $1$ and because $W \in \cap_p U_p$ it is supported in all vertices of dimension one.

If $W$ were not simple its semisimplification $W^{ss}$ would be the sum of simples in $W$, but every simple is supported
in all vertices of dimension one it is not possible that $W^{ss}$ has more than one summand and hence $W$ must be simple.
\end{proof}

\begin{theorem}
Let $(\Q,\gamma)$ be a sincere supermixed quiver setting. 
To check whether $\DRep(\Q,\gamma)$ contains orthogonal simples follow the following algorithm. 
\begin{itemize}
 \item[O0]
$\DRep(\Q,\gamma)$ contains simples if and only if $\DRep(\Q',\alpha')$ does. Here $(Q',\alpha')$ is the
setting obtained by deleting all the antisymmetric arrows of dimension one and all the antisymmetric loops on
symplectic vertices of dimension $2$.
 \item[O1]
If $\Rep(\Q,\gamma)$ does not contain simples 
or there exists a pair of vertices with dimension $1$ such that every path $p:v\to w$ is antisymmetric
then $\DRep(\Q,\gamma)$ does not contain orthogonal $\eps$-simples.
 \item[O2]
If $(\Q,\gamma)$ is basic then we apply lemma \ref{21}.
\item[O3]
If $(\Q,\gamma)$ is not basic then either
\begin{itemize}
\item[A.]
If $(Q,\gamma)$ contains a cycle $c$ with dimension at least $2$,
the setting $(c \cup c^*,1 \text{ or }2)$ has an $\eps$-simple representation of the form in lemma \ref{cykel}.
Consider the local supermixed setting the completed representation $\hat V$. 
$\Rep(\Q,\gamma)$ contains simples if and only if this local setting does.
\item[B.]
If all cycles have dimension one we consider a dimension vector $\beta$ such that $(Q,\beta)$ satisfies the conditions
of lemma \ref{21}, and let $V$ be an orthogonal $\eps$-simple for this setting.
Then $(\Q,\gamma)$ contains orthogonal simples and $(Q,\alpha)$ contains simples if and only if the local setting 
generated by $\hat V$ does.
\end{itemize}
\end{itemize}
\end{theorem}
\begin{proof}
 Step $O0$ works because antisymmetric arrows of dimension one are zero and symmetric loops on $2$-dimensional symplectic vertices
are scalar matrices.
 Step $O1$ follows the line of the first paragraph in the proof of lemma \ref{21}.
 So now we have to prove that after a finite number of O3-steps we end in a basic setting or a non-simple one.
 Step $O3A$ reduces the dimensions of the higher dimensional cycles so after a finite steps there will be no higher-dimensional cycles (see remark \ref{note} and observe that by O0 there are no antisymmetric loops on symplectic vertices with dimension $2$). 
 Every simple setting $(Q,\alpha)$ that has passed step $O0$ will contain a basic subsetting the sum of the dimensions in local subsetting will be strictly lower than the sum $\sum_{v \in Q_0} \alpha_v$, so the algorithm will end after a finite number of $O3B$-steps.
\end{proof}

We will now give an easy application of this algorithm to the easiest quiver settings: the ones with one vertex.
\begin{theorem}
The dualizing space of a supermixedquiver setting with one vertex, $k$ symmetric loops and $l$ antisymmetric loops
contains orthogonal simples if and only if
\begin{enumerate}
 \item the vertex $v$ is orthogonal with dimension $n$ and
\begin{enumerate}
\item if $n=1$
\item if $n=2$, $k+l\ge 2$ and $k\ge 1$,
\item if $n\ge 3$ and $k+l\ge 2$.
\end{enumerate}
 \item the vertex $v$ is symplectic with dimension $n\in 2\N$ and
\begin{enumerate}
\item if $n=2$ and $k\ge 2$
\item if $n=4,6$, $k+l\ge 2$ and $k\ge 1$ or $l \ge 3$
\item if $n\ge 8$ and $k+l\ge 2$.
\end{enumerate}
\end{enumerate}
\end{theorem}
\begin{proof}
First note that if $n>1$ then $k+l\ge2$ because otherwise the underlying quiver setting does not contain simple representations
in $\Rep(Q,\alpha)$ by theorem \ref{simple}. If $n=1$ then every representation is simple and hence every representation in $\DRep(\Q,\gamma)$ is
orthogonal $\eps$-simple.

Now let us first concentrate on the orthogonal case. If we start applying the algorithm the step O3A will do the following
\begin{itemize}
\item If $k\ge 1$ we take $c$ to be a symmetric loop.
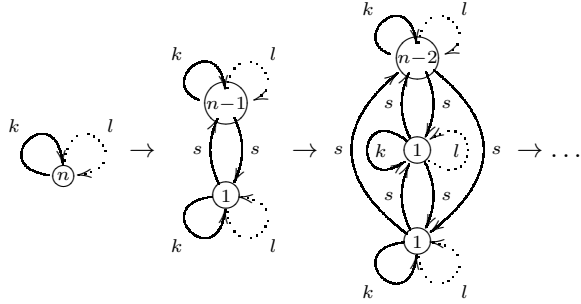
\[
\vcenter{
\xymatrix{\vtx{n}\ar@(l,u)^k\ar@{.>}@(u,r)^l}}
\to
\vcenter{
\xymatrix{\vtx{n-1}\ar@(l,u)^k\ar@{.>}@(u,r)^l\ar@/^/[d]^s\\
\vtx{1}\ar@(l,d)_k\ar@{.>}@(d,r)_l\ar@/^/[u]^s
}}
\to
\vcenter{
\xymatrix{\vtx{n-2}\ar@(l,u)^k\ar@{.>}@(u,r)^l\ar@/^/[d]^s\ar@/^25pt/[dd]^s\\
\vtx{1}\ar@(lu,ld)^k\ar@{.>}@(rd,ru)^l\ar@/^/[u]^s\ar@/^/[d]^s\\
\vtx{1}\ar@(l,d)_k\ar@{.>}@(d,r)_l\ar@/^/[u]^s\ar@/^25pt/[uu]^s
}}
\to \dots
\]
here $s=k+l-1$. So after $n$ steps we get a quiver with only orthogonal vertices of dimension $1$ and $s$ arrows between every pair
of vertices. This setting is basic and hence contains $\eps$-simples.
\item If $k=0$ we must take $c$ to be an antisymmetric loop.
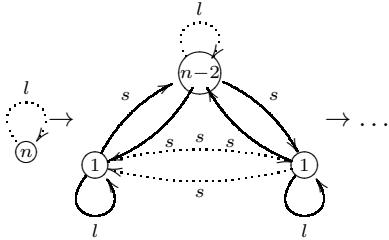
\[
\vcenter{
\xymatrix{\vtx{n}\ar@{.>}@(lu,ru)^l}}
\to
\vcenter{
\xymatrix{&\vtx{n-2}\ar@{.>}@(lu,ru)^l\ar@/^/[ld]^s\ar@/^/[rd]^s&\\
\vtx{1}\ar@{.>}@/^/[rr]^{s}\ar@{->}@(ld,rd)_l\ar@/^/[ur]^s&&\vtx{1}\ar@{->}@(ld,rd)_l\ar@/^/[ul]^s\ar@{.>}@/^/[ll]^{s}
}}
\to \dots
\]
So after $\frac n2$ steps we get a quiver with only orthogonal vertices of dimension $1$ and $l-1$ arrows between every pair
of vertices. The arrows between a vertex and its dual are antisymmetric and can be deleted. Therefore the setting is
only basic if $n>2$.
\end{itemize}

For the symplectic case we can do the same thing but again we have to differentiate between $k=0$ and $k\ge 1$.
\begin{itemize}
 \item 
If $k\ge 1$ the representation we split off corresponding to this symmetric loop in $O3A$ is a general $\eps$-simple. This is the same as for the orthogonal case with $k=0$,
but now there is at least one symmetric arrow between two dual vertices and hence the setting is always basic.
\item
If $k=0$ then $l\ge 2$ and the representation we split of first in $O3A$ is a symplectic $\eps$-simple. Then we can split of 
an $\eps$-simple with corresponding to a cycle that runs through the two cycles (note that such a cycle exists because $k+l-1=l-1\ge 1$).
We can continue this way until the original vertex has dimension $0$. Then we end up with a quiver with one symplectic vertex 
of dimension $2$ and $l((\frac n2)^2-1)+1$ loops of which $\frac{l((\frac n2)^2-1)-(l+1)\frac n2+1}2$ are symmetric. This number can be obtained by using the induction procedure
or more directly by observing that the last quiver setting we obtain corresponds to the local quiver setting of a symplectic simple representation
of the original setting. If the last setting contains an orthogonal representation then $\frac{l((\frac n2)^2-1)-(l+1)\frac n2+1}2$ must be bigger than $1$.
This holds if $l\ge 3$ or if $n\ge 8$.
\end{itemize}

\end{proof}

\bibliographystyle{amsplain}
\bibliography{involution}

\end{document}